\UseRawInputEncoding

\documentclass[english,11pt]{amsart}
\usepackage{amssymb,amsmath,amsthm,mathtools}
\usepackage{srcltx}
\usepackage{babel,graphicx}
\usepackage{color}
\usepackage{hyperref}

\newtheorem{theorem}{Theorem}[section]
\newtheorem{proposition}[theorem]{Proposition}
\newtheorem{corollary}[theorem]{Corollary}
\newtheorem{lemma}[theorem]{Lemma}
\newtheorem{question}[theorem]{Question}

\theoremstyle{definition}
\newtheorem{definition}[theorem]{Definition}
\newtheorem{remark}[theorem]{Remark}

\newcommand{\Z}{\mathbb{Z}}
\newcommand{\Q}{\mathbb{Q}}
\newcommand{\R}{\mathbb{R}}
\newcommand{\C}{\mathbb{C}}

\newcommand{\e}{\epsilon} 

\newcommand{\F}{\mathbb{F}}

\newcommand{\Ocal}{\mathcal{O}}
\newcommand{\hhat}{\widehat{h}}
\newcommand{\Qbar}{\overline{\Q}}
\newcommand{\GL}{\mathrm{GL}}
\newcommand{\SL}{\mathrm{SL}}
\newcommand{\Urm}{\mathrm{U}}
\newcommand{\Tr}{\mathrm{Tr}}
\newcommand{\defeq}{\vcentcolon=}
\DeclarePairedDelimiter\ceil{\lceil}{\rceil}
\DeclarePairedDelimiter\floor{\lfloor}{\rfloor}

\title[A height gap in $\GL_d(\Qbar)$ and almost laws]{A height gap in $\GL_d(\Qbar)$ and almost laws}

\author{Lvzhou Chen, Sebastian Hurtado and Homin Lee}

\subjclass{20F60, 22F50, 37B05, 37C85, 37E10, 57R30.}%

\begin{document}
\maketitle

\begin{abstract}
E. Breuillard showed that finite subsets $F$ of matrices in $GL_d(\overline{\mathbb{Q}})$ generating non-virtually solvable groups have normalized height $\widehat{h}(F) \ge \epsilon_d$, for some positive $\epsilon_d >0$. The normalized height $\widehat{h}(F)$ is a measure of the arithmetic size of $F$ and this result can be thought of as a non-abelian analog of Lehmer's Mahler measure problem. We give a new shorter proof of this result. Our key idea relies on the existence of particular word maps in compact Lie groups (known as almost laws) whose image lies close to the identity element.
\end{abstract}

\section{Introduction}

In \cite{Bre},  E. Breuillard proved what can be considered as a non-abelian version of Lehmer's problem about the Mahler measure  of an algebraic number. He showed that if a set $F \subset \GL_d(\Qbar)$ of invertible matrices with algebraic entries generates a group that is not virtually solvable, then its arithmetic height (a measure of its arithmetic complexity)  is bounded below by an absolute constant $\e_d > 0$ independent of $F$. Some consequences of Breuillard's theorem include the existence of a lower bound for the  exponential growth of a non-virtually solvable group of $\GL_d(\C)$,  a strong version of the classical Tits alternative about existence of free subgroups in linear groups, lower bounds in the girth of finite groups of Lie type and very recently some results about the geometry of arithmetic locally symmetric spaces \cite{FHR}. See \cite{Bre2} for a discussion of some of these applications.

In this article, we will provide a more elementary proof of Breuillard's theorem that avoids the use of Bilu's equidistribution theorem \cite{Bil}, results of Zhang \cite{Zha} about small points on algebraic tori, or results about the geometry of Bruhat--Tits buildings. The key idea in our proof relies on the existence of \emph{word maps} in compact Lie groups whose image lies close to the identity element. These words are known as \emph{almost laws} after Andreas Thom \cite{Tho}; See Section~\ref{almostlaws}.

\subsection{Definitions}

Let $K \subset \Qbar$ be a number field and let $V_K$ be the set of absolute values on $K$ up to equivalence, which are either archimedean (corresponding to real and complex embeddings of $K$ in $\R$ or $\C$) or non-archimedean (corresponding to prime ideals in the ring of integers $\Ocal_K$ of $K$). Let $K_v$ be the corresponding completion of $K$ and define $n_v := [K_v: \Q_p]$, if  $v \vert p$ or $n_v := [K_v: \R]$ if $\R \subseteq K_v$. For a vector $x = (x_1,x_2, \dots, x_d) $ in $K_v^d$, we define $\|x\|_{v} := \max_{i =1}^d |x_i|_v$ if $v$ is non-archimedean and $\|x\|_v := \sqrt{|x_1|_v^2 + \dots + |x_d|_v^2 }$ if $\nu$ is archimedean.

For a matrix $A$ in $\GL_d(K)$, let $\|A\|_v$ be the operator norm.  The \emph{height} of $A$ is defined as
$$h(A) := \frac{1}{[K:\Q]} \sum_{v \in V_K} n_v\log^{+} \|A\|_{v},$$ 
where $\log^{+}(x) = \log \max (|x|, 1)$ for a real number $x$. For a finite set $F$ of matrices in $\GL_d(K)$, the height of $F$ is defined as
$$h(F) := \max_{A \in F} h(A),$$ 
and the \emph{normalized height} is defined by $$\hhat (F) := \lim_{n \to \infty} \frac{h (F^n)}{n},$$ 
where $F^n$ consists of all products $n$ elements in $F$.
The limit above exists because $h(F^n)$ is sub-additive, which also implies $h(F) \geq \hhat(F)$.

Our main result is a new proof of the following theorem due to Breuillard \cite{Bre}:

\begin{theorem}\label{thm: gap} For any positive integer $d$, there exists $\e_d > 0$ such that for any symmetric finite set $F \subset \GL_d(\Qbar)$, either
\begin{enumerate}
\item $F$ generates a virtually solvable group, or
\item $\hhat(F) > \e_d$.
\end{enumerate}
\end{theorem}

\begin{remark}
In section~\ref{3} we give a brief discussion on the possible estimates one can get of $\e_d$ by using our methods and ideas of E. Breuillard and A. Thom.
\end{remark}

\begin{remark} In \cite{Bre}, Breuillard proved this theorem without the assumption that $F$ is symmetric\footnote{symmetric means that if $x \in F$, then $x^{-1} \in F$}. Moreover he showed that $F$ can be conjugated in such a way that their height and normalized height are comparable by a constant only depending on $d$. Our proof does not establish these results. Moreover we use the results of Eskin--Mozes--Oh \cite{EMO} and some lemmas of E. Breuillard \cite{Bre} that allows us to compare the height and normalized height.
\end{remark}

\subsection {Idea of the proof of Theorem \ref{thm: gap}}

We illustrate the idea of the proof in a very particular case. Assume $d = 2$,  $ F \subset \SL_2(\Qbar)$ and assume that all the entries of elements of $F$ are algebraic integers. In this case, the only contribution to the height comes from archimedean places. Let $w \in \F_2$ be an almost law for $\mathrm{SU}(2)$ (see Section \ref{almostlaws} for the definition).  Suppose for simplicity, there exists $A, B \in F$ such that $\Tr(w(A, B)-Id)\neq0$ (see Lemma \ref{lemma: emo}), and let 
\begin{equation}
Q = \prod_{i=1}^k (\Tr( w(A_i,B_i) - Id) ),
\end{equation}
where $k = [K: \Q]$, $A_i=\sigma_i(A)$ and $B_i=\sigma_i(B)$, and $\sigma_i: K \to \C$ are the different embeddings of $K$ in $\C$.

By construction, $Q$ is a nonzero integer and therefore $|Q| \geq 1$.  Let $\delta> 0$ be small enough and let $S$ be the set of $i$'s such that either $\|A_i\|$ or $\|B_i\|$ is  greater than $e^{\delta}$. There are two cases: If $|S| \geq {k/2}$, then one of the heights $h(A)$ or $h(B)$ must be greater than $\frac{1}{2}\delta$ and so $h(F) \geq \frac{1}{2} \delta$ and we are done. If $S \leq k/2$ we have for $i \notin S$,  $A_i$ and $B_i$ are close to a pair of elements of $\mathrm{SU}(2)$ and then for such $i$'s,  $\Tr( w(A_i,B_i) - Id)$ is smaller than say $e^{-10}$ by choosing the almost law appropriately and $\delta$ small enough.

 Therefore from $|Q| \geq 1$, we obtain
 $$ \left|\prod_{i \in S}\Tr( w(A_i,B_i) - Id) \right| \geq  \left| \prod_{i \notin S} \Tr( w(A_i,B_i) - Id)  \right |^{-1} \geq e^{10{k/2}}$$ and this implies that  
the height of $w(A,B)$ must be bounded below by a constant $c > 0$ independent of $k$. Thus $h(F) \geq \frac{1}{|w|} h(F^{|w|}) \geq  \frac{1}{|w|} h(w(A, B)) \geq \frac{c}{|w|}$, where $|w|$ is the length of $w$.

\subsection{Organization of the article} In Section~\ref{1}, we recall definitions and facts about heights, normalized heights, and almost laws. In Section~\ref{2}, we prove Theorem \ref{thm: gap}. In Section~\ref{3} we discuss strategies to obtain explicit estimates about the constant $\e_d$ in terms of $d$.

\subsection{Acknowledgements}
We would like to thank Emannuel Breuillard, Michael Larsen, and Andreas Thom for helpful discussions and suggestions.
We particularly would like to thank Mikolaj Fraczyk, who noticed the connection between a previous version of our results and the work of Breuillard, and also Vladimir Finkelstein for multiple discussions. HL thanks Michael Larsen for pointing out a mistake in Lemma \ref{lemma: emo} from an early draft and thanks Dylan Thurston for his interest and various discussions.

\section{Background}\label{1}
In this section, we give detailed definitions and basic properties of the crucial notions in the statements and proofs of our results: \emph{heights} and \emph{almost laws}.

\subsection{Heights}
For a number field $K \le \overline{\Q}$ and an absolute value $|\cdot|_{v}$ on $K$, define $n_{v}=[K_{v}:\Q_{p}]$ as the degree of the completion $K_{v}$ of $K$ over the closure $\Q_{p}$ of $\Q$ in $K_{v}$. In the case where $\nu $ is archimedean $n_v = 1,2$ depending on whether $\nu$ comes from a complex or real embedding. When $\nu$ is non-archimedean, we normalize the absolute value $|\cdot|_{v}$ on $K_{v}$ so that its restriction to $\Q_{p}$ is the standard absolute value, i.e. $|p|_{v}=1/p$. With such normalization, we have the following product formula:

\begin{theorem}[Product formula]\label{thm:prod}
	Let $K$ be a number field and $V_{K}$ be the set of equivalent classes of absolute values on $K$. Then, for every $x\in K$, 
	\[ \prod_{v\in V_{K}} |x|_{v}^{n_{v}}=1.\]
\end{theorem} 

Let $k$ be a local field of characteristic $0$. Let $\|\cdot\|_{k}$ be the standard norm on $k^{d}$ as in the introduction. We use the same notation for the operator norm on the space $M_{d}(k)$ of $d$ by $d$ matrices with entries in $k$.  Following \cite{Bre}, we define in below the quantities  $\Lambda_k(F)$, $E_{k}(F)$ and $R_k(F)$ for a bounded set $F\subset M_{d}(k)$.
\begin{definition}
	Let $F$ be a bounded subset of matrices in $M_{d}(k)$. We set
	\begin{enumerate}
		\item the \emph{norm} of $F$ as \[\|F\|_{k}=\sup_{g\in F}\|g\|_{k},\]
		\item the \emph{minimal norm} of $F$ as \[E_{k}(F)= \inf_{x\in \textrm{GL}_{d}(\overline{k})}\left\|xFx^{-1}\right\|_{k},\]
		\item the \emph{maximal eigenvalue} of $F$ as \[ \Lambda_{k}(F)=\max\{|\lambda|_{k}, \lambda\in \textrm{spec}(q), q\in F\},\] and
		\item the \emph{spectral radius} of $F$ as \[ R_{k}(F)=\lim_{n\to\infty} \|F^{n}\|_{k}^{1/n}.\]
		
	\end{enumerate}
\end{definition}

	For simplicity, when $k=\C$, we drop $k$ in the notation and denote the above quantities using $\|\cdot\|$, $\Lambda$, and $R$ respectively. If $k$ is understood, sometimes we will use the subscript $v$ for $v\in V_{k}$, such as $\Lambda_{v}$, $E_{v}$ and $R_{v}$, instead of writing subscript $k_{v}$.

With the above definition, we can reformulate the normalized height for a finite set $F\subset M_{d}(k)$ as follows (see \cite[Section 2.2]{Bre}):
\begin{equation*}
	\hhat(F)=\frac{1}{[k:Q]}\sum_{v\in V_K} n_v\log^+ R_v(F)
\end{equation*}
\begin{lemma}[{\cite[Proposition 2.7]{Bre}}]\label{lemma: ineq}
	$\Lambda_k(F)\le R_k(F)$.
\end{lemma}

\begin{lemma}\label{lemma: bound trace by norm}
	Let $|\cdot |_v$ be a non-archimedean place on a field $K$ and equip $M_d(K)$ with the operator norm $\|\cdot \|_v$ induced by the norm $\|x\|_v\defeq \max_i |x_i|_v$ on $K^d$. Then for any $A\in M_d(K)$ we have
	$$|\Tr A|_v\le R_v(A)$$
\end{lemma}
\begin{proof}
	Let $L$ be a finite extension of $K$, such that all the eigenvalues $\lambda_1, \dots, \lambda_d$ of $A$ belong to $L$. There is a unique completion $|\cdot|_w$ extending $|\cdot|_v$. Therefore $\Tr(A)= \lambda_1 + \dots + \lambda_d$ and we have $|\Tr(A)|_v \leq \max_i |\lambda_i|_w = \Lambda_w(A)$. By  Lemma \ref{lemma: ineq}, we have that $\Lambda_w(A)\le R_w(A)$ and it is easily checked that $R_v(A) = R_w(A)$.
\end{proof}

\subsection{Almost laws}\label{almostlaws} Given an element $w$ in a free group $\mathbb{F}_n=\langle x_1,\ldots, x_n\rangle$ of rank $n$, there is a natural \emph{word map} associated to any group $G$
$$w\colon \prod_n G\to G,$$
defined as follows:
Express $w$ as a reduced word with alphabet $\{x_1^{\pm 1},\ldots,x_n^{\pm 1}\}$, and then for any $(g_1,\ldots, g_n)\in \prod_n G$, substitute each $x_i$ by $g_i$ and $x_i^{-1}$ by $g_i^{-1}$.
For example, if $w=[x_1,x_2]  = x_1x_2x_1^{-1}x_2^{-1} \in \mathbb{F}_2$, then $w(A,B) = [A,B]$. 

\begin{definition}
	For a group $G$, a \emph{law} is a nontrivial element $w\in \mathbb{F}_n$ such that the image of $w_G$ is the identity element $1_G$.  Given $\epsilon>0$ and a metric $d$ on $G$, an \emph{$\epsilon$-almost law} is a nontrivial element $w\in \mathbb{F}_n$ such that the image of $G$ lies in an $\epsilon$-neighborhood of $1_G$.
\end{definition}

For instance, the word $w=[x_1,x_2]\in \mathbb{F}_2$ is a law in a group $G$ if and only if $G$ is abelian. In general, groups obeying a law have rather special properties, for instance they contain no nonabelian free subgroups, and the stable commutator length vanishes \cite{Calegari}. Our proof of Theorem \ref{thm: gap} relies on the following result due to A. Thom, which has also been attributed to E. Lindestrauss; see \cite{AvnMei}.

\begin{theorem}[A. Thom \cite{Tho}]\label{thm: almost law} Let $G$ be a compact Lie group and $d_G$ a bi-invariant metric in $G$. For every $\epsilon > 0$, there exists an $\epsilon$-almost law $w_\epsilon\in \mathbb{F}_2$ on $G$, that is, for all $A,B\in G$ we have
	$$ d_{G}(w_{\epsilon}(A, B), 1_G) < \epsilon.$$
\end{theorem}

In $\GL_d(\C)$, an almost law on the compact subgroup $\Urm(d)$ extends to a neighborhood. In below we denote the identity matrix as $I_d$, and let $d_{\GL_d(\C)}(X,Y)\defeq\|X-Y\|$, where $\|A\|$ is the operator norm with respect to the Euclidean metric for any $A \in M_d(\C)$. Note that the restriction of $d_{\GL_d(\C)}$ to $\Urm(d)$ is bi-invariant.
\begin{corollary}\label{cor: almost law} 
	For every $d > 0$ and $\epsilon\in (0,1)$, there exists a nontrivial element $w \in \mathbb{F}_2$ and $\delta=\delta(d,\epsilon) > 0$ such that for any $A, B \in \GL_d(\C)$ satisfying $ \|A\|, \|A^{-1}\|, \|B\|, \|B^{-1}\| < e^{\delta}$ we have	$$ d_{\GL_d(\C)}(w(A, B), I_d) < \epsilon.$$
	Moreover, $w$ can be taken as any $\frac{\e}{2}$-almost law on $\Urm(d)$, and $\delta=\frac{\e}{8|w_{\e/2}|}$.
\end{corollary}
	%
%
\begin{proof}
	Let $w=w_{\e/2}$ be an $\frac{\e}{2}$-almost law as in Theorem \ref{thm: almost law} for $\Urm(d)$ with respect to the restricted metric.
	Then for any $A,B\in\GL_{d}(\C)$, by the singular value decomposition we have $A=P_A\Lambda_A Q_A$ and $B=P_B\Lambda_B Q_B$ where $P_A,Q_A,P_B,Q_B\in\Urm(d)$ and $\Lambda_A,\Lambda_B\in\Delta_d$.
	For $\|A\|, \|A^{-1}\|, \|B\|, \|B^{-1}\|<e^{\delta}$, all diagonal entries of $\Lambda_A,\Lambda_B$ lie in $(e^{-\delta},e^\delta)$.
	
	Let $A'=P_A Q_A$ and $B'=P_B Q_B$, which lie in $\Urm(d)$. Note that $\|A-A'\|=\|\Lambda_A-I_d\|\le e^\delta-1$ and $\|A^{-1}-A'^{-1}\|=\|\Lambda_A^{-1}-1\|\le e^\delta-1$, and similarly for $B$.
	By inserting intermediate words replacing one $A^{\pm 1}$ (resp. $B^{\pm1}$) by $A'^{\pm1}$ (resp. $B'^{\pm1}$) at a time in the word $w_{\e}(A,B)$, we obtain from the triangle inequality that
	$$\|w_{\epsilon/2}(A,B)-w_{\epsilon/2}(A',B')\|\le |w_{\e/2}|\cdot e^{\delta(|w_{\e/2}|-1)}(e^\delta-1).$$
	For $\delta=\frac{\e}{8|w_{\e/2}|}<1$, we have $e^{\delta(|w_{\e/2}|-1)}<e^{\e/8}<2$ and $e^\delta-1<2\delta=\frac{\e}{4|w_{\e/2}|}$.
	Hence
	$$\|w_{\epsilon/2}(A,B)-w_{\epsilon/2}(A',B')\|\le |w_{\e/2}|e^{\delta(|w_{\e/2}|-1)}(e^\delta-1)\le \frac{\e}{4} \cdot e^{\frac{\e}{8}}\le\frac{\e}{2},$$
	and
	$$\|w_{\epsilon/2}(A,B)-I_d\|\le \|w_{\epsilon/2}(A,B)-w_{\epsilon/2}(A',B')\|+\|w_{\epsilon/2}(A',B')-I_d\|\le \epsilon$$
	using the almost law.
\end{proof}

\begin{remark} To calculate explicitly the constant $\e_d$ in the main theorem, we need to compute $|w_{\epsilon/2}|$.
\end{remark}

\section{Proof of Theorem \ref{thm: gap}}\label{2}

We need to prove the uniform lower bound $\hhat(F)\ge \e_d$ whenever $F\subset \GL_d(\Qbar)$ generates a group that is not virtually solvable.

To make use of this assumption on $F$, we need the following lemma, which relies on \cite[Proposition 3.2]{EMO} or \cite[Lemma 4.2]{Bre}.

\begin{lemma}\label{lemma: emo}
	Given any non-trivial word $w \in \F_2$, there is some $n_w$ (only depending on $w$ and $d$) such that for any set $F \subset \GL_d(\C)$ generating a non virtually solvable group, there exists $A,B\in F^{n_w}$ with $\Tr(w(A,B)- I_d)\neq0$.
	
\end{lemma}
\begin{proof}
	Let $\Gamma$ be the group generated by $F$ in $\GL_d(\C)$ and $H$ be the Zariski closure of $\Gamma$. We assumed that $\Gamma$ is not virtually solvable, so $H$ is not virtually solvable. We first show that there exists $X,Y\in \Gamma$ such that $\Tr(w(X,Y)-I_d)\neq 0$.
	
	As $H$ is an algebraic group defined over $\C$, we have the Levi decomposition $H=F\ltimes U$, where $U$ is the unipotent radical of $H$ and $F$ is a reductive Levi subgroup of $H$. Furthermore, we can find a semisimple algebraic group $L=[F,F]$ so that $F$ can be written as an almost direct product of $L$ and the central torus $T$, i.e. $F=L\cdot T$ and $|L\cap T|<\infty$.
	Let $L^{0}$ be the connected component of identity in $L$.  
	
	Assume that $\Tr(w(A,B)-I_d)=0$ for all $A,B \in \Gamma$. We claim that $L^{0}$ is trivial. Indeed, by Borel's theorem \cite{Borel}, the restriction of the word map to $L^{0}$, $w\colon L^{0}\times L^{0}\to L^{0}$ is dominant, i.e. the image is Zariski dense in $L^{0}$. As $R=\{X\in H:\Tr(X)=d\}$ is a Zariski closed subset of $H$, $L^{0} \subset R$. It is easy to see that if a matrix $X$ satisfies $\Tr(X^{m})=d$ for all $m$ then $X$ is a unipotent element. This implies that every element in $L^{0}$  is unipotent so it should be trivial as $L^{0}$ is semisimple. Therefore, $H$ is virtually solvable, as virtually it is an extension of the solvable group $U$ by the abelian group $T$. This is contradicts our assumption, and hence there are $X,Y\in \Gamma$ such that $\Tr(w(X,Y)-I_d)\neq 0$.
	
	For the rest of proof, we can think of $\GL_d(\C)\times \GL_d(\C)$ as a subgroup in $\GL_{2d}(\C)$ diagonally. Let $H'$ be the Zariski closure of $\Gamma\times \Gamma$ in $\GL_{2d}(\C)$. Then the subset 
	\[X= \{(A,B)\in \GL_{d}(\C)\times \GL_{d}(\C): \Tr(w(A,B)-I_d)=0\}\cap H'\subset \GL_{2d}(\C)\] is Zariski closed in $H'$. As we saw previously, $X$ is a proper Zariski closed subset in $H'$. Using the fact that $\Gamma\times \Gamma$ is generated by $F\times F$, \cite[Proposition 3.2]{EMO} or \cite[Lemma 4.2]{Bre} says that there exists $n_{w}\ge 1$ only depending on $d$ and $w$ such that we have 
	\[\{(A,B)\in \Gamma\times \Gamma: A,B \in F^{n}\subset \Gamma\} \nsubseteq X\] as we desired.
\end{proof}

Let $K$ be a number field so that $F\subset \GL_d(K)$. Let $k = [K: \Q]$ and let $\sigma_1, \sigma_2, \dots, \sigma_k$ be all the embeddings of $K$ in $\C$. Fix any $\e\in(0,1)$, choose an $\frac{\e}{2}$-almost law $w=w_{\e/2}$ on $\Urm(d)$ and let $\delta=\frac{\e}{8|w|}$ as in Corollary \ref{cor: almost law}. 
For $w=w(x,y)\in \F_2=\langle x,y\rangle$, let $w'=w([x,y],[x^{-1},y])\in \F_2$. Note that $w'$ is still nontrivial since $[x,y]$ and $[x^{-1},y]$ generate a free subgroup of $\F_2$.
Denote the word length of $w$ by $|w|$. 

For technical reasons later, we need the following improvement of Lemma \ref{lemma: emo} to ensure that $A,B\in\SL_d(\C)$.
\begin{lemma}\label{lemma: promote EMO to SL}
	For the word $w'(x,y)=w([x,y],[x^{-1},y])\in\F_2=\langle x,y\rangle$ as above, there is $n=n_w$ only depending on $w$ and $d$ such that for any finite set $F \subset \GL_d(\C)$ generating a subgroup that is not virtually solvable, there exists $A,B\in F^{n_w}\cap \SL_d(\C)$ with $\Tr(w(A,B)- I_d)\neq0$.
\end{lemma}
\begin{proof}
	Applying Lemma \ref{lemma: emo} to $w'$, we obtain some $n'$ relying only on $w$ and $d$ and $A_1,B_1\in F^{n'}$ such that $\Tr(w(A_1,B_1)-I_d)\neq 0$.
	Let $n=n_w\defeq 4n'$, $A=[A_1,B_1]$ and $B=[A_1^{-1},B_1]$. Then $A,B\in F^n\cap \SL_d(\C)$, and $w(A,B)=w'(A_1,B_1)$ by definition. Hence $\Tr(w(A,B)- I_d)\neq0$.
\end{proof}

Let $n=n_w$ and $A,B\in F^{n_w}\cap \SL_d(\C)$ be as in Lemma \ref{lemma: promote EMO to SL}.

Consider the following quantity: 
\begin{equation}\label{trick}  
	Q \defeq \prod_{i=1}^k (\Tr( w(A_i,B_i) - I_{d}) ),
\end{equation}
where $A_i=\sigma_i(A)$ and $B_i=\sigma_i(B)$.

This is a nonzero rational number since $\Tr(w(A,B)-\mathrm{Id})\neq 0$. 
Let 
$$\e_2=\frac{\log\frac{1}{\e}}{\log \frac{2d}{\e}+\frac{|w|\delta^2}{16d\log \frac{1}{c}}},\quad \text{and}\quad \e_1=\frac{1}{2}\left(\log \frac{1}{\e}-\e_2\log\frac{2d}{\e}\right),$$
where $\delta=\frac{\e}{8|w|}$ as in Corollary \ref{cor: almost law} and $0<c<1$ is the constant (only depending on $d$) from Proposition \ref{constantc} below, which can be chosen as $c=\frac{1}{2d}$.
Note that $\e_2<\log \frac{1}{\e}/\log \frac{2d}{\e}<1$ and thus $\e_1>0$.
Let 
\begin{equation}\label{estimate}
\e_d\defeq \frac{\e_2\delta^2}{32nd\log\frac{1}{c}}=\frac{\delta^2\log\frac{1}{\e}}{32nd\log\frac{2d}{\e}\log\frac{1}{c}+2n|w|\delta^2},
\end{equation}
which is at least at the order of $\frac{1}{n(d^2|w|^2+|w|)}$ as $d\to\infty$ by setting $\e\in(0,1)$ independent of $d$.
The constants $\e_1$ and $\e_2$ are chosen so that
$$\e_d=\frac{\e_2\delta^2}{32nd\log\frac{1}{c}}=\frac{-\e_1 + (1-\e_2) \log \frac{1}{\e}-\e_2 \log 2d}{n|w|}=\frac{\e_1}{n|w|},$$
which are the lower bounds of the normalized height in the analysis below.

We consider the two possibilities in the following two subsections.

\subsection{Case 1: $|Q|\ge e^{-\e_1 k}$} 
In this case, we will use the estimate
\begin{equation}\label{eqn: archimedean}
	\hhat(F)=\frac{1}{k}\sum_{v\in V_K} n_v\log^+ R_v(F)\ge\frac{1}{k}\sum_{i=1}^k \log^+ R_i(F),
\end{equation}
where $R_i(F)=R(\sigma_{i}(F))$ and we simply ignore the non-archimedean places.

Partition the index set $I\defeq\{1,2\cdots, k\}$ as $I=I_S\sqcup I_M\sqcup I_L$, where
$$I_S\defeq\{i: |\Tr(w(A_i,B_i)-I_{d})|<\e\}, \quad I_L\defeq\{i: |\Tr(w(A_i,B_i)-I_{d})|>2d\},$$
and $I_M=I\setminus(I_S\sqcup I_L)$. Note that $I_S\cap I_L=\emptyset$ since $\e<1$.

We further consider two subcases depending on the size of $I_S$.

\subsubsection{Case 1a: $|I_S|>(1-\e_2)k$}
In this case, we have
$$\prod_{i\notin I_S} |\Tr(w(A_i,B_i)-I_d) |=\frac{|Q|}{\left|\prod_{i\in I_S} \Tr(w(A_i,B_i)-I_d)\right|}\ge e^{-\e_1 k}\cdot \e^{-|I_S|}.$$

Taking $\log$ on both sides, we have
$$\sum_{i\notin I_S}\log |\Tr(w(A_i,B_i)-I_d) | \ge -\e_1 k +|I_S|\log \frac{1}{\e} \ge -\e_1 k + (1-\e_2)k \log \frac{1}{\e}.$$

As $|\Tr(w(A_i,B_i)-I_d) |\le 2d$ for all $i\in I_M$, it follows that
\begin{align}\label{eqn: before convexity bound}
	\sum_{i\in I_L}\log |\Tr(w(A_i,B_i)-I_d) | &\ge \sum_{i\notin I_S}\log |\Tr(w(A_i,B_i)-I_d) | - |I_M| \log 2d\nonumber\\ 
	&\ge -\e_1 k + (1-\e_2)k \log \frac{1}{\e}-|I_M|\log 2d.	
\end{align}

\begin{lemma}\label{lemma: convexity bound}
	For any $X\in\GL_d(\C)$, if $|\Tr (X-I_d)|>d$, then the spectral radius $\Lambda(X)\ge \frac{1}{d}|\Tr(X-I_d)|-1$.
	Moreover, if $|\Tr (X-I_d)|>2d$, then $\Lambda(X)\ge \frac{1}{2d} |\Tr(X-I_d)|$.
\end{lemma}
\begin{proof}
	Let $r=\frac{1}{d}|\Tr (X-I_d)|$ and let $\lambda_1,\cdots, \lambda_d\in\C$ be the eigenvalues of $X$. 
	Then for the average $\bar{\lambda}=\frac{1}{d}\sum_{j=1}^d \lambda_j$ we have $r=|\bar{\lambda}-1|$.
	
	It follows that some $\lambda_i$ lies outside the open disk $D$ of radius $r$ around $1\in \C$, since otherwise their average $\bar{\lambda}$ lies in $D$ by convexity, contradicting $r=|\bar{\lambda}-1|$.
	
	Since $r>1$ by our assumption, the disk $D$ contains the open disk $B$ of radius $r-1$ around $0\in\C$.
	Thus $\lambda_i \notin B$ and $\Lambda(X)\ge|\lambda_i|\ge r-1$.
	
	In the case $r>2$, we further have $\Lambda(X)\ge r-1\ge r-\frac{r}{2}=\frac{r}{2}$.
\end{proof}

Applying Lemma \ref{lemma: convexity bound} to the left-hand side of Equation (\ref{eqn: before convexity bound}) we obtain
$$
\sum_{i\in I_L} \log (2d\Lambda(w(A_i,B_i))) \ge -\e_1 k + (1-\e_2)k \log \frac{1}{\e}-|I_M|\log 2d,
$$
\begin{align*}
	\sum_{i\in I_L} \log (\Lambda(w(A_i,B_i))) &\ge -\e_1 k + (1-\e_2)k \log \frac{1}{\e}-(|I_M|+|I_L|)\log 2d\\
	&> -\e_1 k + (1-\e_2)k \log \frac{1}{\e}-\e_2 k\log 2d
\end{align*}
since $|I_M|+|I_L|=|I|-|I_S|<\e_2 k$.

Combining this with Equation (\ref{eqn: archimedean}) and Lemma \ref{lemma: ineq}, as $w(A,B)\in F^{n|w|}$, we conclude
\begin{align*}
	n|w|\cdot \hhat(F)= \hhat(F^{n|w|})&\ge \frac{1}{k}\sum_{i\in I_L} \log (R_i(w(A,B)))\\
	&\ge \frac{1}{k} \sum_{i\in I_L} \log (\Lambda(w(A_i,B_i)))\\
	&\ge -\e_1 + (1-\e_2) \log \frac{1}{\e}-\e_2 \log 2d\\
	&=n|w|\e_d.
\end{align*}
Thus $\hhat(F)\ge \e_d$ as desired in this situation.

\subsubsection{Case 1b: $|I_S|\le (1-\e_2)k$}
Recall that for any $i\notin I_S$, we have $|\Tr(w(A_i,B_i))-I_d|>\epsilon$. 
Since trace is conjugate-invariant, we know that $|\Tr(w(CA_iC^{-1},CB_i C^{-1}))-I_d|>\epsilon$ for all $C\in\GL_d(\C)$.
Thus by Corollary \ref{cor: almost law}, for any $i\notin I_S$ and any $C\in\GL_d(\C)$, at least one of the four matrices $CA_i^{\pm 1}C^{-1},CB_i^{\pm 1}C^{-1}$ has norm no less than $ e^\delta$. Thus we have $E(T_i)\ge e^\delta$ for $T_i\defeq\{A_i^{\pm 1}, B_i^{\pm 1}\}$.

In particular, this implies that for one element among $\{A^{\pm1}, B^{\pm1}\}$, let's say $A$, we have $\|A_i\| \geq e^{\delta} $ for  $\frac{1}{4}\e_2 k$  different $i$'s  and that implies the lower bound $h(F^{n_w}) \geq h(A) \geq \frac{1}{4}\e_2 \delta $, which implies  $$h(F) \geq  \frac{\frac{1}{4}\e_2 \delta}{n_w}.$$

As we want to obtain a bound for the normalized height and not only for the height, we need a way to compare them, to do this we will make use of the following two propositions, which are due to Breuillard. Recall that our $A, B$ lie in $\SL_d(\C)$.

\begin{proposition}[{\cite[Proposition 2.9]{Bre}}]  For any $M\in\Z_+$ and a finite subset $F \subset \SL_d(\C)$, we have 
	$$E(F^M)\ge E(F)^{\sqrt{\frac{M}{4d}}}$$
\end{proposition}

\begin{proposition}[{\cite[Lemma 2.1(b)]{Bre}}]\label{constantc}
	There are uniform constants $c=c(d)$ and $N(d)$ such that for any finite subset $F \subset \GL_d(\C)$ there is $q\in[1,N(d)]$ with
	$$\Lambda(F^{qM})\ge c^q E(F^M)^q,$$
\end{proposition}

The latter proposition can be deduced from an inequality of Bochi \cite[Thereom B]{Bochi}, which was also recently proved by Breuillard in \cite{Bre3} with a good choice of $c=c(d)$ (at the cost of increasing $N(d)$); See \cite[Theorem 5]{Bre3}. The inequality of Breuillard together with \cite[Lemma 1]{Bre3} implies that one can take $c=\frac{1}{2d}$.

These propositions imply that for every $i \notin I_S$, there is some $q$ with
$$\Lambda(T_i^{qM}) \ge c^q E(T_i^M)^q \ge c^qE(T_i)^{q\sqrt{\frac{M}{4d}}} \ge c^q e^{q\delta\sqrt{\frac{M}{4d}}}.$$

Note that for $T\defeq\{A^{\pm 1}, B^{\pm 1}\}$, we have $T\subset F^{n}$ and thus $T_i^{qM}\subset (\sigma_i F)^{nqM}$.
Therefore, using the estimate above, we obtain
$$R_i(F)^{nqM}=R_i(F^{nqM})=R((\sigma_i F)^{nqM})\ge R(T_i^{qM})\ge \Lambda(T_i^{qM})\ge c^q e^{q\delta\sqrt{\frac{M}{4d}}}.$$
That is,
$$\log R_i(F)\ge \frac{1}{nM}\left(\log c+\delta\sqrt{\frac{M}{4d}}\right).$$
Here $\log(c)<0$, so the right hand side is maximized to $\frac{\delta^2}{16nd\log(1/c)}$ when $M=16d\log^2 c/\delta^2$.
As $M$ is an integer, we should take $M=\floor*{16d\log^2 c/\delta^2}$ or $\ceil*{16d\log^2 c/\delta^2}$.
Since $16d\log^2 c/\delta^2$ is sufficiently large, in either case we get
$$\log R_i(F)\ge \frac{\delta^2}{32nd\log\frac{1}{c}}.$$

Hence by Equation (\ref{eqn: archimedean}), we have
$$\hhat(F)\ge \frac{1}{k} \sum_{i\notin I_S}\log^+ R_i(F)\ge \frac{\e_2 \delta^2}{32nd\log\frac{1}{c}}=\e_d.$$


\subsection{Case 2: $|Q| < e^{-\e_1 k}$}

In this case, for $\alpha\defeq \Tr(w(A,B)-I_d)$, we can apply the product formula Theorem \ref{thm:prod} so that
$$\prod_{v\in V_K^f}|\alpha|^{n_v}_v=\frac{1}{|Q|}>e^{\e_1 k},$$
where $V_K^f$ denotes the non-archimedean places of $K$.

Since $|\cdot |_v$ is an ultrametric for all $v\in V_K^f$, by Lemma \ref{lemma: bound trace by norm} we have 
$$|\alpha|_v \le \max\{ | \Tr(w(A,B)) |_v, | \Tr(I_d) |_v\} \leq  \max \{ R_v(w(A, B)), 1 \}.$$
Hence for all $v\in V_K^f$ we have
$$\log|\alpha|_v\le \log^+  R_v(w(A,B))$$

Therefore, this inequality and the product formula together imply
\begin{align*}
	\hhat(F)=\frac{1}{n|w|}\hhat(F^{n|w|}) &\geq \hhat(w(A, B))\\
	&\ge\frac{1}{n|w|k} \sum_{v\in V_K^f} \log^+ R_v(w(A,B))\\
	&\ge \frac{1}{n|w|k} \sum_{v\in V_K^f} \log |\alpha|_v\\
	&\ge \frac{1}{n|w|k} (\e_1 k)\\
	&=\frac{\e_1}{n|w|}\\
	&=\e_d.
\end{align*}


\section{Remarks about explicit estimates of the height gap $\e_d$}\label{3}

One can construct examples showing that the height gap  $\e_d \leq \frac{c}{d}$ for some $c>0$ independent of $d$. A natural question is to determined the actual order of $\e_d$ and Breuillard has asked the following:

\begin{question}Is $\e_d \geq \frac{1}{cd^c}$ for some $c>0$?
\end{question}

We will explain how our method might lead to obtain some explicit bounds for this constant.\\

From equation (\ref{estimate}), we have that 
$$\e_d \geq \frac{\delta^2\log\frac{1}{\e}}{32n_wd\log\frac{2d}{\e}\log\frac{1}{c}+2n_w|w|\delta^2},$$ 
where $c$ is the constant appearing in Proposition \ref{constantc} and $\delta = \frac{\e}{8|w|}$. Recently Breuillard proved that $c$ can be taken to $\frac{1}{2d}$ by \cite[Theorem 5 and Lemma 1]{Bre3}. Therefore by taking $\e = 1/4$, we see that $\e_d$ is at least at the order of
$$\frac{1}{n_{w}(d^2|w|^2+|w|)},$$
where $|w|$ is the word length of a $\frac{1}{4}$-almost law $w$ on $\Urm(d)$ and the constant $n_w$ comes from the escape of the hypersurface $V$ defined by $$ \Tr( w (X,Y) - \text{Id}) = 0$$ as in Lemma \ref{lemma: promote EMO to SL}.

One can obtain a bound of the order of $10^{10^d}$ for $|w|$ by considering the almost laws described by Thom in \cite{Tho}. It is nonetheless quite likely that a much shorter almost law $w$ exists. Thom and Breuillard pointed out to us the fact that Kozma and Thom showed in \cite{KozTho} that for the symmetric group $S_d$ (which in some ways behave similarly to $\mathrm{SU}(d)$ when $d$ is large), there exists a law of order $e^{C \log^{4}(d) \log\log(d)}$, which assuming Babai's conjecture could be improved to be of the order $e^{C \log(d) \log\log(d)}$, which is quite close to be polynomial. 

The constant $n_w$ when computed from the generalized Bezout's Theorem as in \cite{EMO} seems to be quite large.  It was suggested by Breuillard that there seems to be another way of estimating $n_w$ by considering an appropriate finite quotient of the group generated by $F$ (for example by moding out a prime ideal of the ring where $F$ is defined ), and then showing that in this finite quotient one can escape from the hypersurface $V$ quickly. The advantage is that there are various results about the diameter of simple groups. Michael Larsen suggested the use of a strong approximation theorem by Weisfeiler \cite{Wei} to find the appropriate quotient. 

To implement this idea, one must make sure the reduction of $V$ to this finite group of Lie type is a proper subset. The dimension of the variety $V$ is bounded by $d|w|$ and so the reduction of $V$ can be proved to be a proper subset  via the Schwartz--Zippel lemma, or via the Lang--Weil estimates, provided that the finite quotient is large enough (the prime ideal giving rise to the finite quotient must have covolume larger than the degree of $V$, at least).

Assuming Babai's conjecture about the diameter of finite groups one then might expect that $n_w$ could be taken of the order $(|d\log(|w|))^{O(1)}$ and using results towards Babai conjecture this can possibly be taken to be of the order $e^{(|d\log(|w|))^{O(1)}}$, see \cite{BDH}.


\end{document}